\documentclass{amsart}
\usepackage{amssymb,latexsym}
\usepackage{amsmath}
\usepackage{amscd}
\usepackage{graphicx}
 \usepackage{color}
\usepackage{enumerate}
\numberwithin{equation}{section}
\theoremstyle{plain}
 \newtheorem{theorem}{Theorem}[section]
 \newtheorem{lemma}[theorem]{Lemma}

 \newtheorem{corollary}[theorem]{Corollary}
\theoremstyle{definition}

 \newtheorem{example}[theorem]{Example}

\renewcommand \epsilon {{\boldsymbol\varepsilon}}
\newcommand \trap {D}
\newcommand \goodset {H}
\newcommand \trngle {T}

\newcommand \circel[2] {\textup{Circ}(#1,#2)}
\newcommand \disk[2] {\textup{Disk}(#1,#2)}
\newcommand \dist[2]{\textup{dist}(#1,#2)}

\newcommand \ctre[1] {#1^{\mathord{\bullet}}}

\newcommand \ulambda{\underline\lambda}

\newcommand \eeqref[1] { \overset{\textup{\eqref{#1}}}=  }
\newcommand \ssty[1]{\scriptscriptstyle{#1}}
\newcommand \hconc[2] {\boldsymbol\chi_{#1,  \ctre{#2}}}
\newcommand \hthet[2] {\boldsymbol\chi_{#1,#2}}
\newcommand \loosin {\overset{\ssty{\textup{loose}}}\subset}
\newcommand \reangle [2]{{\textup{Ang}(#1,#2)}}
\newcommand \rhullo [1] {\textup{Conv}_{\real^{#1}}} 
\newcommand\set [1]{\{#1\}}
\newcommand \tuple [1] {\langle #1 \rangle}
\newcommand \pair [2] {\tuple{#1,#2}}
\newcommand \real {\mathbb R}
\newcommand \preal {\real^{2}}
\newcommand \nothing [1] {}
\newcommand \red [1] {{\color{red}#1\color{black}}}
\newcommand \tbf[1] {\textbf{#1}}  

\begin{document}
\title[Convexity and circles]
{An easy way to a theorem of Kira Adaricheva and Madina Bolat on convexity and circles}

\author[G.\ Cz\'edli]{G\'abor Cz\'edli}
\email{czedli@math.u-szeged.hu}
\urladdr{http://www.math.u-szeged.hu/\textasciitilde{}czedli/}
\address{University of Szeged\\ Bolyai Institute\\Szeged,
Aradi v\'ertan\'uk tere 1\\ Hungary 6720}

\thanks{This research was supported by
NFSR of Hungary (OTKA), grant number K 115518}

\begin{abstract} 
Kira Adaricheva and Madina Bolat have recently proved that if $U_0$ and $U_1$ are circles in a triangle with vertices $A_0,A_1,A_2$, then there exist $j\in \set{0,1,2}$ and $k\in\set{0,1}$ such that $U_{1-k}$ is included in the convex hull of $U_k\cup(\set{A_0,A_1, A_2}\setminus\set{A_j})$. We give a short new proof for this result,  and we point out that a straightforward generalization for spheres fails.
 
\end{abstract}

\subjclass {Primary 52C99, secondary 52A01}

\dedicatory{Dedicated to the eighty-fifth birthday of B\'ela Cs\'ak\'any}

\keywords{Convex hull, circle, sphere, abstract convex geometry, anti-exchange system, Carath\'eodory's theorem, carousel rule}

\date{\red{May 17, 2017, extended version}}

\maketitle

\section{Aim and introduction}
\subsection*{Our goal}
The real $n$-dimensional space and the usual convex hull operator on it will be denoted by $\real ^n$ and $\rhullo n$. That is, for a set $X\subseteq \real^n$ of points,  $\rhullo n(X)$ is the smallest convex subset of $\real^n$ that includes $X$. In this paper, the Euclidean distance $(\sum_{i=1}^n (X_i-Y_i)^2)^{1/2}$ of $X,Y\in\real^n$ is denoted by $\dist X Y$. 
For $P\in \real ^2$ and $0\leq r\in \real$, the \emph{circle}
of center $P$ and radius $r$ will be denoted by
\begin{equation*}
\circel P r:=\set{X\in\real^2:  \dist P X = r}. 
\end{equation*}
Our aim is to give a new proof of the following theorem. Our approach is entirely different from and shorter than the original one given by  Adaricheva and Bolat~\cite{kaczg}. Roughly saying, the novelty is that instead of dealing with several cases, we prove that the ``supremum of good cases" implies the result for all cases.

\begin{theorem}[{Adaricheva and Bolat~\cite[Theorem 3.1]{kabolat}}]\label{thmmain} Let  $A_0,A_1,A_2$ be points in the plane. If $U_0$ and $U_1$ are circles such that $U_i\subseteq\rhullo 2(\set{A_0,A_1,A_2})$ for $i\in\set{0,1}$, then there exist subscripts $j\in \set{0,1,2}$ and $k\in\set{0,1}$ such that 
\begin{equation}
U_{1-k} \subseteq \rhullo 2\bigl( U_k\cup(\set{A_0,A_1, A_2}\setminus\set{A_j})\bigr).
\label{eqthmmain}
\end{equation}
\end{theorem}

Notably enough, Adaricheva and Bolat
~\cite[Theorem 5.1]{kabolat} states even more than \cite[Theorem 3.1]{kabolat}; we formulate their more general result as follows.

\begin{corollary}[{Adaricheva and Bolat~\cite[Theorem 5.1]{kabolat}}]\label{corolmain}
If $C_0$, $C_1$, $C_2$, $U_0$, and $U_1$ are circles in the plane such that $U_i\subseteq\rhullo 2(C_0\cup C_1\cup C_2)$ for $i\in\set{0,1}$, then $U_{1-k} \subseteq \rhullo 2\bigl( U_k\cup\bigcup(\set{C_0,C_1, C_2}\setminus\set{C_j})\bigr)$ holds for some $j\in \set{0,1,2}$ and $k\in\set{0,1}$.
\end{corollary}

Note that Adaricheva and Bolat~\cite{kabolat} call the property stated in this corollary for circles the ``Weak Carousel property''. 
Note also that \cite{kabolat} gives a new justification to Cz\'edli and Kincses~\cite{czgkj}, because Theorem 5.2 and Section 6 in \cite{kabolat} yield that
the almost-circles in \cite{czgkj} cannot be replaced by circles. Also,  \cite{kabolat} motivates 
Cz\'edli~\cite{czgCharc}   and Kincses~\cite{Kincs}.
This paper is self-contained. 
For more about the background of this topic, the reader may want, but need not, to see, for example, Adaricheva and Nation~\cite {kirajbbooksection} and \cite{kajbn}, Cz\'edli~\cite{czgcircles}, Edelman and Jamison~\cite{edeljam}, Kashiwabara,  Nakamura, and  Okamoto~\cite{kashiwabaraatalshelling}, Monjardet~\cite{monjardet}, and   Richter and Rogers~\cite{richterrogers}.

The results of Adaricheva and Bolat~\cite{kabolat}, that is, Theorem~\ref{thmmain} and Corollary~\ref{corolmain} above, and our easy approach raise the question whether the most straightforward generalizations hold for $3$-dimensional spheres.
In Section~\ref{sectexampl}, which is a by-product of our method in some implicit sense,  we give a negative  answer.

\section{Homotheties and round-edged angles}

\subsection{A single circle}
For $0<r\in\real$ and $F,P\in\real^2$ with  $\dist F P>r$, let 
\begin{equation}
\reangle F{\circel P r}\text{ be the  grey-filled area in Figure~\ref{figegy};}
\label{eqcrlDsKc}
\end{equation}
it is called the \emph{round-edged angle} determined by its \emph{focus} $F$ and \emph{spanning circle} $\circel P r$. Note that $\reangle F{\circel P r}$ is not bounded from the right and $F$ is outside both $\circel P r$ and $\reangle F{\disk P r}$. 
Note that  $\reangle F{\disk P r}$ includes its boundary, which consists of a circular arc called the \emph{front arc} and two half-lines.

\begin{figure}[ht] 
\centerline
{\includegraphics[scale=1.0]{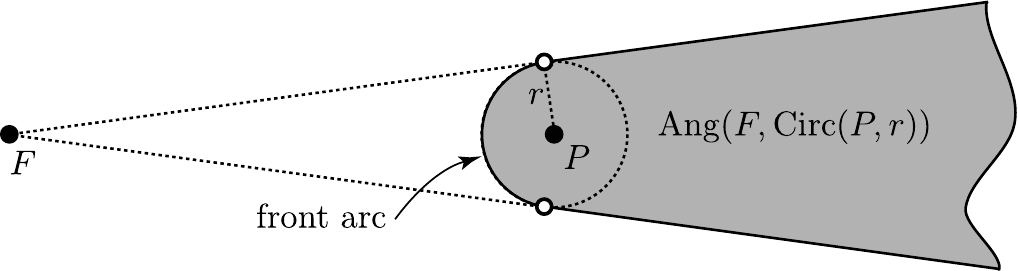}}
\caption{Round-edged angle
\label{figegy}}
\end{figure}%

\subsection{Externally perspective circles}
First,  recall or define some easy concepts and notations.  
For topologically closed convex sets $W_1,W_2\subseteq \preal$, we will say that 
\begin{equation}
\text{$W_1$ is \emph{loosely included} in $W_2$, in notation, $W_1\loosin W_2$,}
\label{eqlooSinCl}
\end{equation}
if every point of $W_1$ is an internal point of $W_2$.
Given $P\in\preal$ and $0\neq\lambda\in\real$,  the \emph{homothety} with (homothetic) center $P$ and ratio $\lambda$ is defined by  
\begin{equation}
\text{
$\hthet P\lambda \colon \real^2\to \real^2$ by  $X \mapsto PX\ulambda:=(1-\lambda)P+\lambda X$.}
\label{eqtxthmThdF}
\end{equation}
We will not need negative ratios $\lambda$ and
we use the Polish notation for the \emph{barycentric operation} $\ulambda$. Homotheties are similarity transformations. In particular, they map the center of a circle to the center of its image.  If $C_1$ and $C_2$ are circles and $C_2=\hthet P\lambda(C_1)$ such that $P$ is (strictly) outside both $C_1$ and $C_2$ (equivalently, if $P$ is outside $C_1$ or $C_2$) and $0<\lambda\in\real$, then $C_1$ and $C_2$ will be called \emph{externally perspective circles}. 
Clearly, if $C_1$ and $C_2$ are of different radii and none of them is inside the other, then $C_1$ and $C_2$ are externally perspective,  $P$ is the intersection point of their external tangent lines, and $\lambda$ is the ratio of their radii.

\begin{lemma}\label{lemmaperspCs}
Let $\circel {P_1}{r_1}$ and $\circel {P_2}{r_2}$ be externally perspective circles in 
the plane with center $F$ of perspectivity such that  $0<r_2<r_1$; see Figure~\ref{figmbAT}. If $G$ is a point on the line segment  $[F,P_2]$ such that  $r_2<\dist{G}{P_2}<\dist{F}{P_2}$,
then   $\reangle F{\circel {P_1}{r_1}}\loosin\reangle G{\circel {P_2}{r_2}}$; see \eqref{eqlooSinCl} and Figure~\ref{figmbAT}.
\end{lemma}

\begin{figure}[ht] 
\centerline
{\includegraphics[scale=1.0]{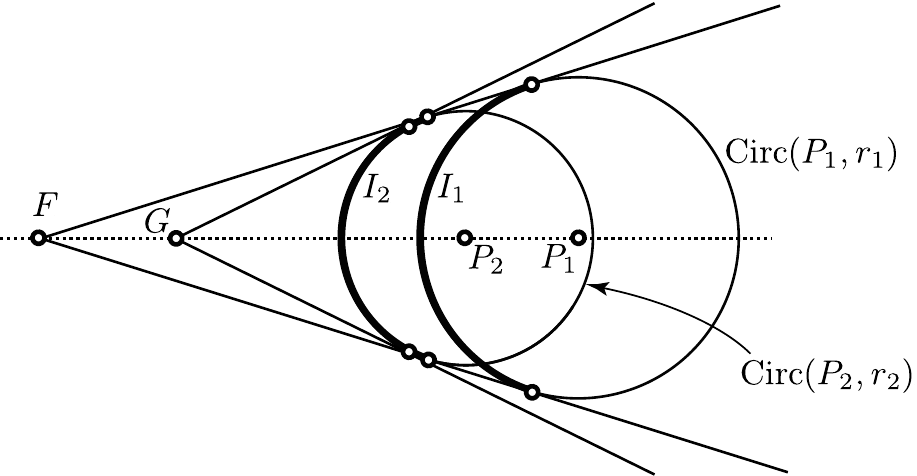}}
\caption{Illustration for Lemma~\ref{lemmaperspCs}
\label{figmbAT}}
\end{figure}%

\begin{proof} Clearly, $\circel {P_1}{r_1} =  \hthet F\lambda (\circel {P_2}{r_2})$ with 
$\lambda=r_1/r_2>1$. The external tangent lines of our circles  intersect at $F$.  Since  $\hthet F \lambda $ preserves tangency, it maps the circular arc  $I_2$ of $\circel {P_2}{r_2}$ between the tangent points onto the circular arc $I_1$ of $\circel {P_1}{r_1}$  between the images of these tangent points; see the thick arcs in  Figure~\ref{figmbAT}.
Hence, $I_2$ is strictly on the left of $I_1$ in  the figure, implying the lemma.
\end{proof}

\begin{lemma}\label{lemmaBlowUp} If $\lambda,\mu\in\real\setminus\set0$, $F, Q\in\preal$, and $R=\hthet F\lambda(Q)$, then,  composing maps from right to left, 
$\hthet R\mu \circ \hthet F\lambda = \hthet F\lambda\circ \hthet Q\mu$. 
\end{lemma}

\begin{proof} 
$\hthet F\lambda\circ \hthet Q\mu \circ\hthet F\lambda^{-1}$ is clearly a homothety of ratio $\mu$ that fixes $R$. So this homothety is $\hthet R\mu$, which implies the lemma.
\end{proof}

\begin{lemma}\label{lemmaintrtgntrC}
If $\lambda>1$ and $C_0$ and $C_1$ are internally tangent circles with center points $\ctre C_0$ and $\ctre C_1$, respectively, 
then either one of $\hconc\lambda{C_0}(C_0)$ and $\hconc\lambda{C_1}(C_1)$ is in the interior of the other, or $C_0=C_1$.
\end{lemma}

\begin{figure}[ht] 
\centerline
{\includegraphics[scale=1.0]{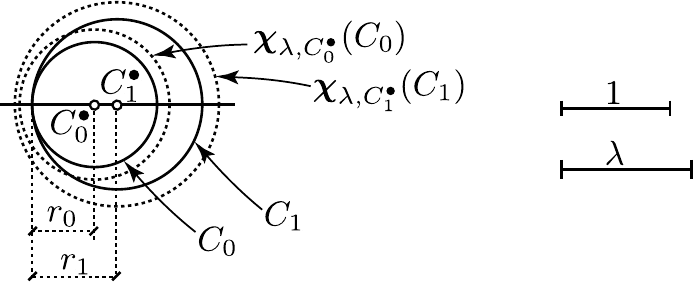}}
\caption{Illustration for Lemma~\ref{lemmaintrtgntrC} 
\label{fightDcdgT}}
\end{figure}%

\begin{proof} We can assume that the radii $r_0$ and $r_1$ are distinct, say,  $r_0<r_1$; see Figure~\ref{fightDcdgT}. The distance $d:=\dist{\ctre C_0}{\ctre C_1}$ is $r_1-r_0$.  Since $\lambda r_1=\lambda(r_0+d)>\lambda r_0+d$, $\hconc {\lambda}{C_{0}}(C_{0})$ is in the interior of $\hconc {\lambda}{C_{1}}(C_{1})$, as required.
\end{proof}

The following lemma resembles the 2-Carousel Rule in Adaricheva~\cite{adaricarousel}.

\begin{lemma}\label{lemmafpnTlsl}
Let $A_0$, $A_1$, and $A_2$ be non-collinear points in the plane. If  $B_0$ and $B_1$ are distinct \emph{internal} points of $\rhullo 2(\set{A_0,A_1,A_2})$, then there exist $j\in\set{0,1,2}$ and $k\in\set{0,1}$ such~that 
\[\set{B_{1-k}}\loosin \rhullo 2\bigl(\set{B_k}\cup (\set{A_0,A_1,A_2}\setminus\set{A_j})\bigr).
\]
\end{lemma}

\begin{proof} Since the triangle  $\rhullo 2(\set{A_0,A_1,A_2})$ is clearly of the form
\begin{equation}
\begin{aligned}
\rhullo 2(\set{B_0,A_1,A_2})
 \cup {} \rhullo 2(\set{A_0,B_0,A_2}) \cup \rhullo 2(\set{A_0,A_1,B_0}),
\end{aligned}
\label{eqdTgrndSs}
\end{equation}
$B_1$ belongs to at least one of the triangles in \eqref{eqdTgrndSs}. If one of these three  triangles, say, 
$\rhullo 2(\set{B_0,A_1,A_2})$, contains $B_1$ as an internal point, then we let $k=0$ and $j=0$. Otherwise, there is a $j'\in\set{0,1,2}$ such that  the line segment  $[B_0,A_{j'}]$ contains $B_1$ in its interior, and we can clearly let $k=1$ and $j=j'$. 
\end{proof}

\section{Proving Theorem~\ref{thmmain} with analytic tools}

\begin{proof}[Proof of Theorem~\ref{thmmain}]
If $A_0$, $A_1$, and $A_2$ are collinear points, then the circles are of radii 0 and \eqref{eqthmmain} holds trivially (even without $U_k$ on the right). Hence, in the rest of the proof, we assume that $A_0$, $A_1$, and $A_2$ are non-collinear points. We let
\[
\trngle:=\rhullo2{(\set{A_0,A_1,A_2})}.
\]
Let $P_i$ and $r_i$ denote the center and the radius of $U_i$ from the theorem. Note that
\begin{equation}
\text{$r_1=0$ implies \eqref{eqthmmain}, by \eqref{eqdTgrndSs} applied for $B_0\in U_0$ and $B_1=P_1$;}
\label{eqtxtfgzTnB}
\end{equation}
and similarly for $r_0=0$. Therefore, we will assume that none of $r_0$ and $r_1$ is zero. From now on, we prove the theorem by way of contradiction. That is, we assume that $U_0$ and $U_1$ are circles satisfying the assumptions of Theorem~\ref{thmmain}, $r_0r_1>0$, but \eqref{eqthmmain} fails. For $0\leq \xi\leq 1$ and $k\in\set{0,1}$, we denote $\circel {P_k}{\xi\cdot r_k}$ by $U_k(\xi)$. Let 
\begin{equation}
\begin{aligned}
\parbox{9.5cm}{$\goodset:=\{\eta\in[0,1]: (\forall \zeta\in[0,\eta])\,(\exists k\in\set{0,1})\,
(\exists j\in\set{0,1,2})$
 such that $
U_{1-k}(\zeta) \subseteq \rhullo 2\bigl( U_k(\zeta)\cup(\set{A_0,A_1, A_2}\setminus\set{A_j})\bigr)\}$.
}
\end{aligned}
\label{eqtxtksrGdStfB}
\end{equation}
In other words, $\goodset$ consists of those $\eta$ for which $U_0(\zeta)$, $U_1(\zeta)$, $A_0$, $A_1$, and $A_2$ satisfy the theorem for all $\zeta$ in the closed interval $[0,\eta]\subseteq [0,1]\subseteq \real$.  
For brevity, we let
\begin{equation}
\parbox{10cm}{$
W(j,k,\zeta):=\rhullo 2\bigl( U_k(\zeta)\cup(\set{A_0,A_1, A_2}\setminus\set{A_j})\bigr)$; 
then $\goodset:=\{\eta\in[0,1]: (\forall\zeta\in[0,\eta])\,(\exists k)\,(\exists j)\,(U_{1-k}(\zeta) \subseteq W(j,k,\zeta)\}$.
}
\label{eqdhmnCbmTm}
\end{equation}
By \eqref{eqtxtfgzTnB}, $0\in\goodset$.
Since $U_k(1)=U_k$, for $k\in\set{0,1}$, our indirect assumption gives that $1\notin \goodset$. Clearly, if $0\leq \eta_1\leq\eta_2\leq 1$ and $\eta_2$ belongs to $\goodset$, then so does $\eta_1$; in other words, $\goodset$ is an order ideal of the poset $\tuple{[0,1],\leq}$. 
From now on, 
\begin{equation}
\text{let $\xi$  denote the supremum of $\goodset$.}
\label{eqtxtXidef}
\end{equation} 
We are going to show that 
\begin{equation}
\text{$\xi\in\goodset$, whereby $\xi$ is actually the maximum of $\goodset$, and $\xi > 0$.}
\label{eqtxtxiisMax}
\end{equation}
Since $r_0,r_1>0$ and $P_0$ and $P_1$ are internal points of the triangle $\trngle$, it follows from Lemma~\ref{lemmafpnTlsl} that 
$\xi>0$.
In order to prove the rest of  \eqref{eqtxtxiisMax} by way of contradiction, suppose that $\xi\notin \goodset$. However, for each $i$ such that $\lceil 1/\xi\rceil <  i\in\mathbb N$, in short, for each \emph{sufficiently large} $i$, $\xi-1/i \in \goodset$. Hence, for each sufficiently large $i$, we can pick a $k_i\in \set{0,1}$ and a $j_i\in\set{0,1,2}$ such that  $U_{1-k_i}(\xi-1/i)\subseteq W(j_i,k_i,\xi-1/i)$; see \eqref{eqdhmnCbmTm}. Since $\set{0,1}\times\set{0,1,2}$ is finite, one of its pairs, $\pair k j$, occurs infinitely many times in the sequence of pairs $\pair{k_i}{j_i}$.  Thus, there exist a $k\in\set{0,1}$, a $j\in\set{0,1,2}$, and an infinite set $I\subseteq \mathbb N$ of sufficiently large integers $i$ such that 
\begin{equation}
\text{for all }i\in I,\text{ we have that }U_{1-k}(\xi-1/i)\subseteq W(j,k,\xi-1/i).
\label{eqmTzTmYyy}
\end{equation}
Since, for all $\eta$ and $\zeta$, $0\leq\eta\leq\zeta$ implies  $W(j,k,\eta)\subseteq W(j,k,\zeta)$, \eqref{eqmTzTmYyy} yields that 
\begin{equation}
\text{for all }i\in I,\text{ we have that }U_{1-k}(\xi-1/i)\subseteq W(j,k,\xi).
\label{eqmTzgZHYys}
\end{equation}
Next, let $X$ be an arbitrary point of the circle $U_{1-k}(\xi)$. Denote by $X_i$ the point
$\hthet{P_{1-k}}{(\xi-1/i)/\xi}(X)$; it belongs to $U_{1-k}(\xi-1/i)$. 
Less formally, we obtain $X_i$ as the intersection of $U_{1-k}(\xi-1/i)$ with the line segment connecting $X$ and $P_{1-k}$. 
As $i\in I$ tends to $\infty$, $X_i\to X$. 
Combining this with $X_i\in U_{1-k}(\xi-1/i)$ and \eqref{eqmTzgZHYys}, we obtain that $X$ is a \emph{limit point} (AKA accumulation point or cluster point)  of $W(j,k,\xi)$.  
The convex hull of a compact subset of $\real^n$ is  compact; see, for example, Proposition 5.2.5 in Papadopoulos~\cite{papadopoulos}. Hence, $W(j,k,\xi)$ from  \eqref{eqdhmnCbmTm} is a compact set; whereby it contains its limit point, $X$. Thus, since $X$ was an arbitrary point of $U_{1-k}(\xi)$, we conclude that 
$U_{1-k}(\xi)\subseteq W(j,k,\xi)$. By \eqref{eqdhmnCbmTm}, this proves \eqref{eqtxtxiisMax}.

\begin{figure}[ht] 
\centerline
{\includegraphics[scale=1.0]{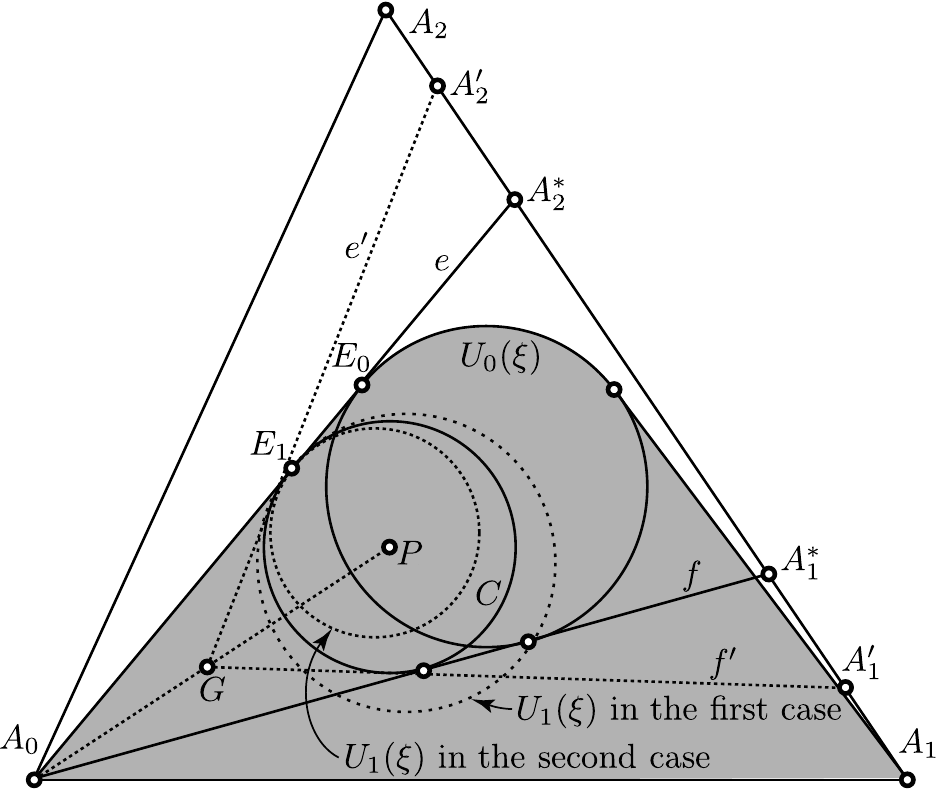}}
\caption{Illustration for \eqref{eqPxidfnt} 
\label{fightmThGdrrgtT}}
\end{figure}%

Since $\xi\in\goodset$, we can assume that the indices are chosen so that $U_1(\xi)$
is included in the grey-filled ``round-backed trapezoid'' 
\begin{equation}
\trap(\xi):= \rhullo2(\set{A_0,A_1}\cup U_0(\xi)); \text{ see Figure~\ref{fightmThGdrrgtT}.}
\label{eqPxidfnt}
\end{equation}
If $U_1(\xi)$ was included in the interior of $\trap(\xi)$, then there would be a (small) positive $\epsilon$ such that $U_1(\xi+\delta )\subseteq \trap(\xi)\subseteq \trap(\xi+\delta)$ for all $\delta\in(0,\epsilon]$  and $\xi+\epsilon$ would belong to $\goodset$, contradicting \eqref{eqtxtXidef}. Therefore, $U_1(\xi)$ is tangent to the boundary of $\trap(\xi)$. Since $\xi<1$ and $U_1(1)= U_1$ is still included in the triangle $T$, $U_1(\xi)$ cannot be tangent to the side $[A_0,A_1]$ of $T$.  If $U_1(\xi)$ was tangent to the  back arc of the  ``round-backed trapezoid''  $\trap(\xi)$ and so to $U_0(\xi)$, then one of $U_0=U_0(1)$ and  $U_1=U_1(1)$ would be in the interior of the other by Lemma~\ref{lemmaintrtgntrC}, and this would contradict the indirect assumption that \eqref{eqthmmain} fails. 
Hence $U_1(\xi)$ is tangent to one of the ``legs" of $\trap(\xi)$; this leg is an external tangent line $e$ of the circles $U_1(\xi)$ and $U_0(\xi)$ through, say, $A_0$; see Figure~\ref{fightmThGdrrgtT}. The corresponding touching points will be denoted by $E_1$ and $E_0$; see the figure. Let $\lambda:=\dist{A_0}{E_1}/\dist{A_0}{E_0}$; note that $0<\lambda<1$. By well-known properties of homotheties, the auxiliary circle 
\begin{equation}
C:=\hthet{A_0}\lambda(U_0(\xi)),\text{ with center }P:=\hthet{A_0}\lambda(P_0),
\label{eqtxtsddtGfhNmW}
\end{equation}
touches $e$ and, thus, $U_1(\xi)$ at $E_1$. Let $f$ denote the other tangent of $U_0(\xi)$ through $A_0$.  Let $A_1^\ast$ and $A_2^\ast$ be the intersection points of $f$ and $e$ with the line through $A_1$ and $A_2$, respectively. Since $U_0(1)=U_0$ is also included in $\trngle$ and $U_0(\xi)$ is a smaller circle concentric to $U_0$, both $A_1^\ast$ and $A_2^\ast$ are in the interior of the line segment $[A_1,A_2]$. By continuity, we can find a point $G$ in the interior of the line segment $[A_0,P]$ such that $G$ is outside $C$ and $G$ is so close to $A_0$ that the tangent lines $e'$ and $f'$ of $C$ through $G$ intersect the  line segments $[A_2^\ast,A_2]$ and $[A_1,A_1^\ast]$ at some of their \emph{internal} points, which we denote by $A_2'$ and $A_1'$, respectively. Since the ``round-backed trapezoid'' $\rhullo2(\set{A_1',A_2'}\cup C)$ is clearly the intersection of the round-edged angle $\reangle G C$ and one of the  half-planes determined by the line through $A_1'$ and $A_2'$, we obtain from Lemma~\ref{lemmaperspCs} that $U_0(\xi)\loosin \rhullo2(\set{A_1',A_2'}\cup C)$. 
Combining this with the obvious $\rhullo2(\set{A_1',A_2'}\cup C)\subseteq \rhullo2(\set{A_1,A_2}\cup C)$,  we obtain that $U_0(\xi)\loosin \rhullo2(\set{A_1,A_2}\cup C)$. 
Thus, we conclude that there exists a (small) positive $\epsilon$ in the interval $ (0, 1-\xi)$ such that 
\begin{equation}
U_0(\xi+\delta)\subseteq \rhullo2(\set{A_1,A_2}\cup C)\text{ for all }\delta\in (0,\epsilon].
\label{eqxhmTvbnW}
\end{equation}
Let $r$ be the radius of $C$. Depending on $r$, there are two cases. First, if $r_1>r$, then $C$ is inside $U_1(\xi)$ and, consequently, also in $U_1(\xi+\delta)$, whereby  \eqref{eqxhmTvbnW} leads to 
\[
U_0(\xi+\delta)\subseteq \rhullo2(\set{A_1,A_2}\cup C) \subseteq \rhullo2(\set{A_1,A_2}\cup U_1(\xi+\delta))
\]
for all $\delta\in(0,\epsilon]$. This gives that $\xi+\epsilon\in\goodset$, contradicting \eqref{eqtxtXidef}.

Second, let $r_1\leq r$. Now 
$U_1(\xi)$ coincides with or  is inside $C$.  By Lemma~\ref{lemmaintrtgntrC}, 
\begin{equation}
\text{for all $\mu>1$, $\hthet P\mu(U_1(\xi))$ coincides with or is inside $\hthet P\mu(C)$.}
\label{eqtxthgnBrsX}
\end{equation}
Clearly, $C\loosin T$, since so is $U_0(\xi)$.
Hence, we can choose a (small) positive $\delta$ such that  $\hthet{P_0}\mu(U_0(\xi))=U_0(\xi\mu)$ and $\hthet{P}\mu(C)$ are loosely included in $T$ for every $\mu\in[1,1+\delta]$. 
Furthermore, for every $\mu\in[1,1+\delta]$,
\begin{equation}
\begin{aligned}
\hthet P\mu(C)
&\eeqref{eqtxtsddtGfhNmW} 
\hthet P\mu\bigl(\hthet{A_0}\lambda(U_0(\xi))\bigr)
\cr 
&\overset{\textup{Lemma~\ref{lemmaBlowUp}}}=
\hthet{A_0}\lambda(\hthet {P_0}\mu(U_0(\xi)))
= \hthet{A_0}\lambda(U_0(\xi\mu)).
\end{aligned}
\label{eqdnbhRhjNq}
\end{equation}
Since $0<\lambda<1$, it follows  that
\begin{equation}
\hthet P\mu(C) \eeqref {eqdnbhRhjNq}
\hthet{A_0}\lambda(U_0(\xi\mu))\in\rhullo2(\set{A_0}\cup U_0(\xi\mu)),\text{ whence}
\label{eqhfWnPZwWt}
\end{equation}
\begin{align*}
U_1(\xi\mu)&=\hthet P\mu(U_1(\xi))
\overset{\eqref{eqtxthgnBrsX}}\subseteq
\rhullo2(\hthet P\mu(C))
\overset{\eqref{eqhfWnPZwWt}}\subseteq
\rhullo2\bigl(\set{A_0}\cup U_0(\xi\mu)\bigr).
\end{align*}
Since this holds for all $\mu\in[1,1+\delta]$, we conclude that $\xi(1+\delta)\in\goodset$.
This contradicts \eqref{eqtxtXidef}, 
completing the proof of Theorem~\ref{thmmain}. 
\end{proof}

\begin{figure}[ht] 
\centerline
{\includegraphics[scale=1.0]{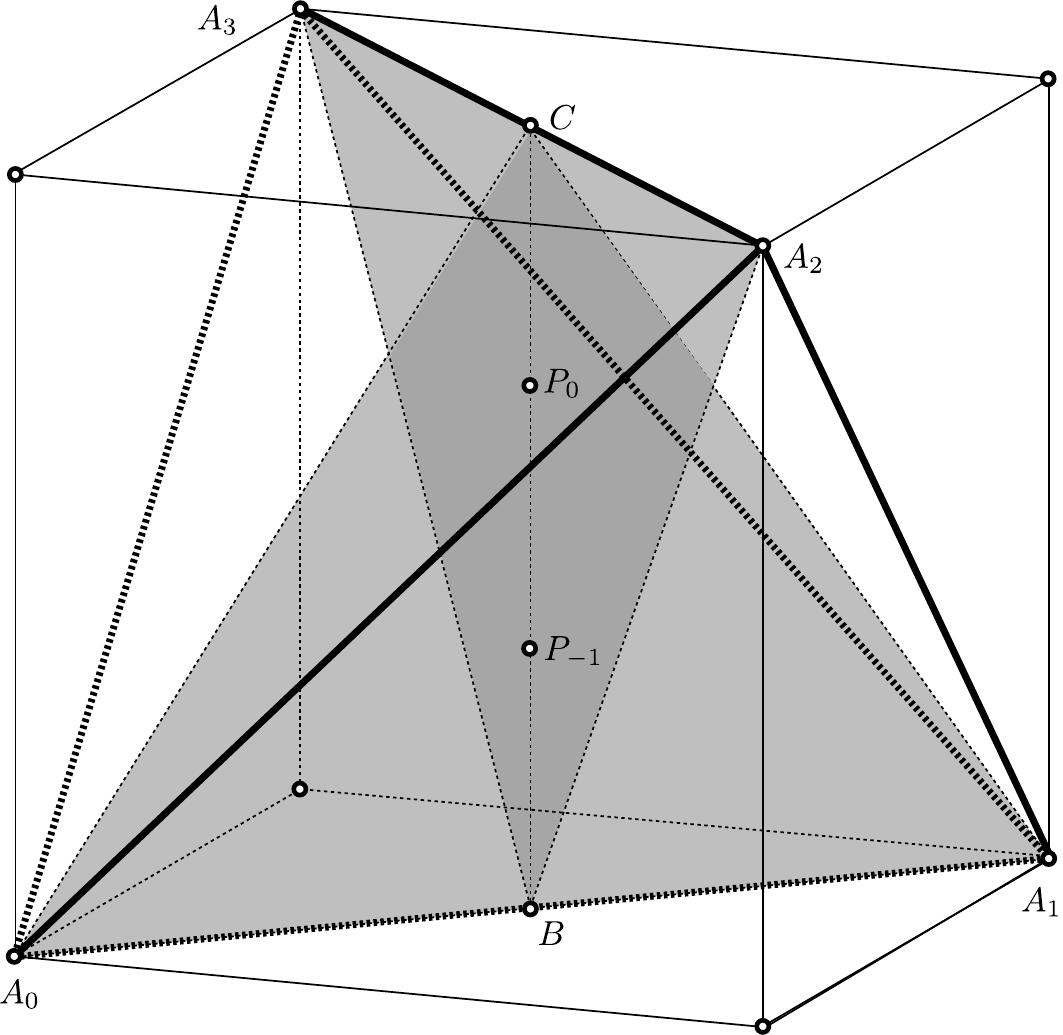}}
\caption{A regular tetrahedron in a cube}
\label{figksKbntDr}
\end{figure}%

\begin{figure}[ht] 
\centerline
{\includegraphics[scale=1.0]{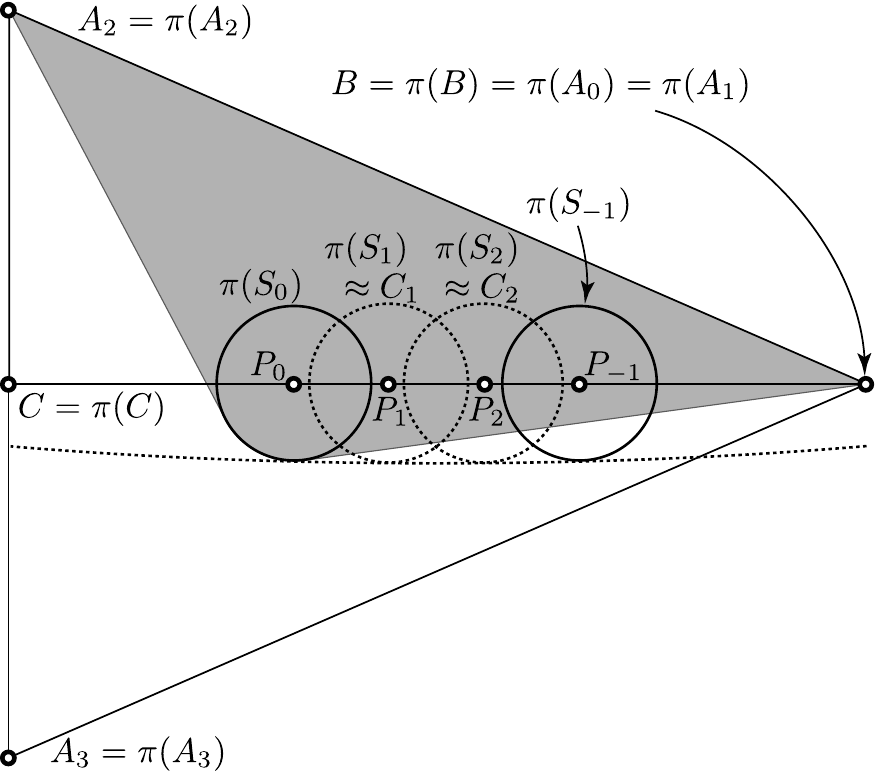}}
\caption{The $\pi$-images of our spheres}
\label{eqshpCrprCtn}
\end{figure}%

\section{Examples}\label{sectexampl}
This section explains why we have been unable to generalize Theorem~\ref{thmmain} for spheres so far. In our first example, one can change $\set{-1,0}$ and $-1-k$ to  $\set{0,1}$ and $1-k$, respectively; we have chosen  $\set{-1,0}$ and $-1-k$ for a technical reason.

\begin{example}  Let $A_0,\dots,A_3$ be the vertices of a regular tetrahedron as well as some vertices of a cube; see Figure~\ref{figksKbntDr}. Let $B$ and $C$ be the middle points of the line segments $[A_0,A_1]$ and $[A_2,A_3]$, respectively, and let $P_{-1}$ and $P_0$ divide $[B,C]$ into three equal parts as the figure shows.  Finally, let $S_{-1}$ and $S_0$ be spheres in the interior of the tetrahedron $\rhullo3(\set{A_0,\dots,A_3})$ with centers $P_{-1}$ and $P_0$ and of the same positive radius. Then, for all 
 $j\in \set{0,1,2,3}$ and $k\in\set{-1,0}$, 
\begin{equation}
S_{-1-k} \nsubseteq \rhullo 3\bigl( S_k\cup\bigcup(\set{A_0,A_1, A_2,A_3}\setminus\set{A_j})\bigr).
\label{eqsPrhflsUV}
\end{equation}
\end{example}

\begin{proof}By symmetry, it suffices to show 
\eqref{eqsPrhflsUV} only for $j=3$. First, let $k=0$.
We denote by $\pi$ the orthogonal projection of $\real^3$ to the plane containing $A_2$, $A_3$ and $B$. Suppose for a contradiction that $S_{-1}\subseteq\rhullo3(S_0\cup A_0\cup A_1\cup A_2)$;
this inclusion is preserved by $\pi$.
Since $\pi$ commutes with the formation of convex hulls and the disk $\pi(S_{-1})$ 
is not included in  
$\rhullo2(\pi(S_{0})\cup\pi(\set{A_0,A_1,A_2}))$, 
the grey-filled area in  Figure~\ref{eqshpCrprCtn}, which is a contradiction.
Second, if $k=-1$, then the argument is essentially the  same but the grey-filled area in Figure~\ref{eqshpCrprCtn} has to be changed.
\end{proof}

\begin{example}\label{examplsjdT}
For $t\in\set{3,4,5,\dots}$, add $t-2$ additional spheres to the previous example in the following way. Let $P_1,\dots, P_{t-2}$ divide the line segment $[P_0,P_{-1}]$ equidistantly; see Figure~\ref{eqshpCrprCtn} for $t=4$. This figure contains also a circular dotted arc with a sufficiently large radius; its center is far above the triangle. Besides the boundary circles of the disks $\pi(S_0)$ and $\pi(S_{-1})$ from the previous example, let $C_1$,\dots, $C_{t-2}$ be additional circles with centers $P_1,\dots,P_{t-2}$ such that all the (little) circles  are tangent to the dotted arc; this idea is taken from Cz\'edli~\cite[Figure 5]{czgcircles}. 
For $i\in\set{1,\dots, t-2}$, let $S_i$ be the sphere obtained from $C_i$ by  rotating it around the line through $B$ and $C$.
Note that $\pi(S_i)\approx C_i$ in Figure~\ref{eqshpCrprCtn} means that the circle $C_i$ is the boundary of the \emph{disk} $\pi(S_i)$.
Now,  for all 
 $j\in \set{0,1,2,3}$ and $k\in\set{-1,0,\dots,t-2}$, $S_{k}$ is \emph{not} a subset of
\begin{equation*}
\rhullo 2\bigl(
\bigcup(\set{S_{-1},S_0,\dots,S_{t-2}}\setminus\set{S_k})
\cup\bigcup(\set{A_0,A_1, A_2,A_3}\setminus\set{A_j})\bigr),
\end{equation*}
while all the  $S_k$ are still included in the tetrahedron $\rhullo3(\set{A_0,\dots,A_3})$. 
\end{example}

\begin{proof} Combine  the previous proof and Cz\'edli~\cite[Example 4.3]{czgcircles}.
\end{proof}

\end{document}